\documentclass[12pt]{amsart}

\usepackage[utf8]{inputenc}
\usepackage{amsfonts, amssymb, amsmath, amsthm, enumerate, multicol, xspace, array, multirow, graphicx, bbm,mathtools,color,soul,tikz}
\usetikzlibrary{math}

\newcommand{\N}{\mathbb{N}}

\newcommand{\CT}{\mathcal{T}}

\newcommand{\Id}{\text{Id}}

\newcommand{\CF}{\mathcal{F}}

\newcommand{\ep}{s}

\newcommand{\rest}{\upharpoonright}

\renewcommand{\phi}{\varphi}
\newcommand{\actson}{\mathrel{\curvearrowright}}

\newcommand{\eps}{\varepsilon}

\newcommand{\Mgraph}{M}
\newcommand{\MgraphV}{{V(M)}}
\newcommand{\MgraphE}{{E(M)}}

\DeclareMathOperator{\wdeg}{wdeg}

\DeclareMathOperator{\Aut}{Aut}

\DeclareMathOperator{\LIM}{LIM}

\DeclareMathOperator{\Image}{Im}

\DeclareMathOperator{\proj}{proj}

\definecolor{amethyst}{rgb}{0.6, 0.4, 0.8}
\definecolor{ao}{rgb}{0.0, 0.5, 0.0}
\definecolor{carrotorange}{rgb}{0.93, 0.57, 0.13}
\definecolor{carmine}{rgb}{0.59, 0.0, 0.09}

\newenvironment{propositionproof}{
  
  \begin{proof}
}{\end{proof}}

\newtheorem{thm}{Theorem}[section]

\newtheorem{prop}{Proposition}[section]

\newtheorem{cor}[prop]{Corollary}

\newtheorem{lem}{Lemma}[section]

\newtheorem{numclaim}[prop]{Claim}

\theoremstyle{definition}
\newtheorem{definition}{Definition}[section]

\begin{document}

\title[Borel Perfect Matchings in Quasi-transitive Graphs]{Borel Fractional Perfect Matchings in Quasi-transitive Amenable Graphs}
\author[Sam Murray]{Sam Murray}
\address{McGill University Department of Mathematics and Statistics, 805 Sherbrooke St W, Montreal, Quebec H3A 1G3}
\email{samuel.murray@mail.mcgill.ca}

\maketitle

\begin{abstract}
    We show that if a locally finite Borel graph with quasi-transitive amenable components admits a fractional perfect matching, it will admit a Borel fractional perfect matching. In particular, if a countable amenable quasi-transitive graph admits a fractional perfect matching then its Bernoulli graph admits a Borel fractional perfect matching. 
\end{abstract}

\section{Introduction}\label{sec:intro}

The object of descriptive combinatorics is to take familiar graph theoretic problems on a Borel graph and study what happens when one requires that the solution be Borel, Baire measurable with respect to a compatible topology, or measurable with respect to a complete Borel measure. A problem of particular interest in the descriptive context due to its application to measurable circle squaring \cite{meascirclesquaring} and more recently Borel circle squaring \cite{borelcirclesquaring} is the perfect matching problem: when does a Borel graph admit a descriptive perfect matching?  A result of Laczkovich shows that there exist 2-regular graphs that admit a perfect matching but no Borel, Baire measurable, or measurable perfect matching \cite{Laczkovich_nomatch}. On the other hand, Bowen, Kun and Sabok proved that if a graphing is regular, bipartite, one-ended, and hyperfinite, then it admits a measurable perfect matching \cite{bowen2022perfect}. In a subsequent paper, Bowen,  Poulin and Zomback showed the same thing holds in the Baire measurable context \cite{bowen2022oneended}.\\ 

One technique used in this paper is to consider the linear relaxation of the perfect matching problem, called the fractional perfect matching problem, where we allow the values on the edges to be in the interval $[0,1]$ instead of just $\{0,1\}$. Once a solution is obtained to this relaxed problem, it can then be rounded to a solution for the integral problem. The existence of measurable fractional perfect matchings, specifically measurable fractional perfect matchings that are non-integer everywhere, is what was used in the paper of Bowen, Sabok, and Kun to prove the existence of measurable perfect matchings \cite{bowen2022perfect}.  A similar idea was used in a follow up paper on the Gardner conjecture \cite{bowen2022uniformgardnerconjecturerounding}.\\

A natural question to ask is when do Borel or measurable fractional perfect matchings exist? In the measure context, an argument due independently to Lov\'{a}sz \cite{Lovászbook} and to Cie\'sla and Sabok \cite{hallsabelain} shows that any locally finite hyperfinite
graphing that admits a perfect matching will admit a measurable fractional perfect matching. Outside of the measure context, little can be said: in an upcoming paper by Bernshteyn and Weilacher \cite{felixbairetree}, they construct a polynomial growth Borel forest on a Polish space that has no Borel fractional perfect matching, even after throwing away an invariant meager set. In contrast to this result, we show:
\begin{thm}\label{thm:mainthm}
    If $G$ is a locally finite Borel graph that is componentwise quasi-transitive and amenable, then $G$ has a Borel fractional perfect matching.
\end{thm}

In particular, this shows that quasi-transitive Borel graphs with polynomial growth will have Borel fractional perfect matchings. This also gives the following corollary about the Bernoulli graphs:
\begin{cor}
    If $G$ is a countable quasi-transitive amenable graph with a fractional perfect matching, then its Bernoulli graph will admit a Borel fractional perfect matching.
\end{cor}

We will also give a result that substitutes amenability for a tiling condition. The proof is based on a similar idea as Theorem \ref{thm:mainthm}.
\begin{thm}\label{thm:tileable_graph}
    If $G$ is a locally finite Borel bipartite graph satisfying Halls theorem, and $G$ is Borel symmetrical tilable, then $G$ admits a Borel fractional perfect matching.
\end{thm}

\section{Preliminaries and Notation}\label{sec:prelim}

If $X$ is a standard Borel space then $[X]^{<\N}$ is the standard Borel space of all finite subsets of $X$. If $(a_n)_{n\in\N}$ is a sequence of elements in a compact Hausdorff space, let $\LIM_{n} a_n$ be the unique ultralimit of the sequence with respect to a fixed nonprinciple ultrafilter on $\N$.\\

For a graph $G = (V,E)$ and $v\in V$, $h = (x,y)\in E$ let: 
\begin{align*}
    N_V(v) &= \{u\in V: u E v\} \subseteq V\\
    N_E(v) &= \{e\in E: e = (v,u) \text{for some } u\in V\} \subseteq E\\
    B_n(v) &= \{u\in V: d(u,v)\leq n\} \subseteq V\\
    B_n(h) &= B_n(x)\cup B_n(y) \subseteq V\\
\end{align*}
For any $v\in V$, let $N_E(v) = N_E(\{v\})$, $N_V(v) = N_V(\{v\})$, $B_n(v) = B_n(\{v\})$. \\

In this paper, we only consider locally finite graphs. A countable graph $G$ is called \textbf{quasi-transitive} if there are only finitely many orbits of the action $\Aut(G)\actson V(G)$.\\

A \textbf{rooted graph} $(G,r)$ is a graph with a chosen root vertex $x$. A \textbf{labeling} of a graph $G$ is a map $l:V(G) \to [0,1]$. A \textbf{rooted labeled graph} is a rooted graph along with a labeling, $(G,r,l)$.\\

Given a countable graph $G = (V,E)$, the \textbf{Bernoulli graph} of $G$ is the space of isomorphism classes of rooted labeled graphs $(G,r,l)$ such that $(G,r_0,l_0)$ and $(G,r_1,l_1)$ are adjacent if there exists a neighbor of $r_0$ which is isomorphic to $(G,r_1,l_1)$ as a labeled rooted graph. The Bernoulli graph of $G$ is a Borel graph, because isomorphism of locally rooted labeled graphs is a smooth equivalence relation. If $G$ is quasi-transitive, its Bernoulli graph is componentwise quasi-transitive. If $G$ is amenable, its Bernoulli graph is amenable.\\\\

The \textbf{vertex isoperimetric constant} of a graph $G$ is $$\iota_V(G) = \inf\{\frac{|N_V(S)|}{|S|}: S\text{ finite connected subset of }V(G)\}$$
A countable graph is \textbf{amenable} if it has $\iota_V(G) = 0$.\\

\begin{figure}[ht]
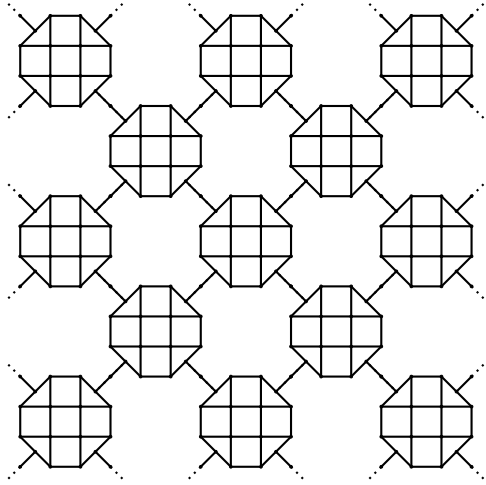

    \centering

    \caption{An amenable quasi-transitive graph.}
    \label{fig:amenable_quasi}
\end{figure}

A locally compact second countable group $\Gamma$ with Haar measure $\mu$ is \textbf{amenable} if for every compact set $A\subseteq \Gamma$ and every $\eps > 0$, there exists compact $B\subseteq \Gamma$ such that $\frac{\mu(AB\Delta B)}{\mu(B)} < \eps$. 

\begin{thm}[Salvatori, 1992]
 If $G$ is a quasi-transitive countable graph, then $G$ is amenable implies $\Aut(G)$ is amenable.
\end{thm}

A locally finite Borel graph $G$ on standard Borel space $X$ is \textbf{componentwise quasi-transitive} if each connected component is quasi-transitive. $G$ is \textbf{componentwise amenable} if each connected component is amenable.\\

A \textbf{fractional perfect matching} on $G$ is a symmetric function $f:E\to[0,1]$ satisfying $\sum_{e\in N^E(v)} f(e) = 1$ for all $v\in V$. If $G$ is a locally finite Borel graph, then $f$ is a Borel fractional perfect matching if it is a fractional matching that is a Borel function from $E$ to $[0,1]$.\\

A \textbf{ multigraph} is a tuple: $$\Mgraph = (\MgraphV,\MgraphE,\ep)$$ where $\MgraphV$ and $\MgraphE$ are sets representing the vertices and edges, and $\ep:\MgraphE\to [\MgraphV]^2\cup [\MgraphV]^1$ is the \textbf{endpoint map} which maps an edge $e$ to the set of its endpoints. Note that this definition allows a multigraph to have loops.
For a multigraph $\Mgraph = (\MgraphV,\MgraphE,\ep)$ and $v\in \MgraphV$ let: 
\begin{align*}
    N_\MgraphE(v) &= \{e\in \MgraphE: v\in \ep(e)\}\\
    N_\MgraphV(v) &= \bigcup(\ep(N_\MgraphE(v)))
\end{align*}
 Say that a Borel multigraph $\Mgraph$ has \textbf{countable multiplicity} if $\ep$ is countable to one.\\

An \textbf{induced sub-multigraph} of a multigraph $\Mgraph$ on a subset of vertices $V'\subseteq \MgraphV$ is the multigraph with vertices $V'$, edges $E' = s^{-1}([V']^2\cup[V']^1)$, and endpoint map $s\rest E'$.\\

A \textbf{homomorphism} from a multigraph $\Mgraph = (\MgraphV,\MgraphE,\ep_\Mgraph)$ to a multigraph $H = (V(H), E(H),\ep_H)$ is a pair of maps $\phi_{V}:\MgraphV \to V(H)$ and $\phi_{E}:\MgraphE \to E(H)$ that commute with the endpoint maps, so for any edge $e \in \MgraphE$, we have $\phi_V(\ep_M(e)) = \ep_H(\phi_E(e))$.\\

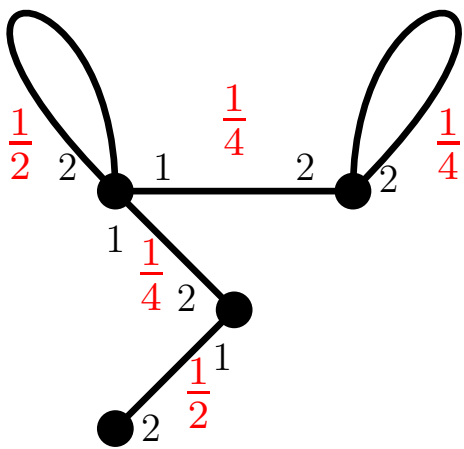
\begin{figure}[ht]
\label{fig:weighted_multigraph}
\resizebox{10cm}{!}{
\begin{tikzpicture}

\draw[black, thick] (0,0) .. controls (1,1) and (0,1) .. (0,0);
\node at (.15,.05)[scale = .4, black]{2};

\draw[black, thick] (-1,0) .. controls (-2,1) and (-1,1) .. (-1,0);
\node at (-1.2,.1)[scale = .4, black]{2};

\draw[black, thick]  (-.5,-.5) -- (-1,0);

\node at (-.2,.1)[scale = .4, black]{2};
\node at (-.8,.1)[scale = .4, black]{1};

\node at (-1,-.2)[scale = .4, black]{1};
\node at (-.7,-.45)[scale = .4, black]{2};

\node at (-.55,-.7)[scale = .4, black]{1};
\node at (-.85,-1)[scale = .4, black]{2};

\draw[black, thick] (0,0) -- (-1,0);
\draw[black, thick] (-1,-1) -- (-.5,-.5);

\filldraw[black] (0,0) circle (2pt);
\filldraw[black] (-1,0) circle (2pt);
\filldraw[black] (-.5,-.5) circle (2pt);
\filldraw[black] (-1,-1) circle (2pt);

\node at (-.5,.3)[scale = .6, red]{$\frac14$};
\node at (.4,.2)[scale = .6, red]{$\frac14$};
\node at (-1.4,.2)[scale = .6, red]{$\frac12$};
\node at (-.85,-.35)[scale = .6, red]{$\frac14$};
\node at (-.65,-.85)[scale = .6, red]{$\frac12$};

\end{tikzpicture}}
\centering
    \caption{A half-edge weighted multigraph component with a fractional perfect matching. The smaller black numbers represent the half-edge weighting, the red numbers are the fractional perfect matching.}
\label{fig:weighted_multigraph}
\end{figure}
An \textbf{half-edge weighting} of a multigraph $\Mgraph = (\MgraphV,\MgraphE,\ep)$ is a map $m:\MgraphV\times \MgraphE\to \N$ with $m(v,e) > 0$ iff $e\in N_E(v)$ (see Figure \ref{fig:weighted_multigraph}). A \textbf{fractional perfect matching} on a weighted multigraph $G$ with a edge weighting $m$ is a function $f:E\to [0,1]$ satisfying: $$\sum_{e\in N_E(v)} f(e)m(v,e) = 1$$ For all $v\in V$. A \textbf{weighted multigraph} $\Mgraph = (\MgraphV,\MgraphE,\ep, m)$ is a multigraph along with an edge weighting. Let the \textbf{weighted degree} of $v\in V(M)$ be $\wdeg_M(v) = \sum_{e\in N_E(v)} m(v,e)$. Note that any graph $G=(V,E)$ can be made into a weighted multigraph $G$ by taking $\ep(e)$ to be the set of endpoints of $e$, taking $m(v,e) = 1$ for every vertex $v$ and edge $e$ incident to $v$, and setting $\Mgraph = (V, E,\ep, m)$. In this case $\wdeg_M = \deg_G$.\\

A multigraph $\Mgraph = (\MgraphV,\MgraphE,\ep)$ is called a \textbf{Borel multigraph} if $\MgraphV$ and $\MgraphE$ are standard Borel spaces and $\ep$ is a Borel map. If $\Mgraph$ is additionally equipped with a Borel edge weighting $m:V\times E\to \N$, call it a \textbf{Borel weighted multigraph}. Note if $\Mgraph$ is a Borel multigraph, then $N_\MgraphE(v)$ is always a Borel set. If $\Mgraph$ has countable multiplicity then $N_\MgraphV(v)$ is also always Borel.\\


A \textbf{graph with half-edges} is a graph $G$ along with a set $H(G)$ and an assignment $t: H(G) \to V(M)$. Call elements of $H(G)$ half-edges. Say $v\in \MgraphV$ is incident to $h\in H$ if $t(h) = v$. Let $N_E(v) = \{e\in E(G)\cup H(G):\ e\text{ is incident to }x\}$.\\

For a graph $G$ and a set of vertices $S\subseteq V(G)$, the \textbf{induced subgraph with half-edges} is defined as the induced graph on $S$ along with a half-edge for each edge in the boundary of $S$, where the assignment $t$ sends an edge in the boundary to the vertex incident to it in $S$. We denote the induced subgraph with half-edges on $S$ by $N_{\frac{1}{2}}(S)$.\\

A \textbf{perfect matching} of a graph with half-edges $m:E(G)\cup H(G)\to\{0,1\}$ such that every $v\in V$ is incident to exactly one element of $m^{-1}({1})$.\\

We can similarly adapt the definition of fractional perfect matching to graphs with half-edges: a \textbf{fractional perfect matching} on a graph with half-edges $G$ is a function $f:E(G)\cup H(G)\to[0,1]$ symmetric on $E$ satisfying $\sum_{e\in N_E(v)} f(e) = 1$ for all $v\in V$. Note that if $G$ is a bipartite graph with half-edges, then if it satisfies Hall's condition, ie $|N_E(S)|\geq |S|$ for all finite $S\subseteq G$, then $G$ has a perfect matching.\\

For a graph $G$, let the \textbf{perfect matching kernel} $K\subseteq E(G)$ be the set of edges $e$ such that there is a perfect matching of $G$ containing $e$ and a perfect matching of $G$ not containing $e$. In particular, if $e$ is in the perfect kernel then there exists a fractional perfect matching of $G$ which is non-integer on $e$.\\

Let $F$ be a finite graph with half-edges. Say that a Borel graph $G$ on standard Borel space $X$ is \textbf{Borel $F$-tileable} if there exists a Borel $\CT\subseteq [X]^{<\N}$ consisting of disjoint $G$ connected components such $\bigcup \CT = X$ and for every $S\in \CT$ we have $N_\frac12(S)\cong F$.\\

A graph with half-edges $F$ is \textbf{half-edge transitive} if for each pair of half-edges $h,f \in H(F)$ there exists an automorphism on $F$ sending $h$ to $f$. A Borel graph is \textbf{Borel symmetrically tilable} if it is Borel $F$-tileable for a half-edge transitive $F$.

\begin{figure}[ht]
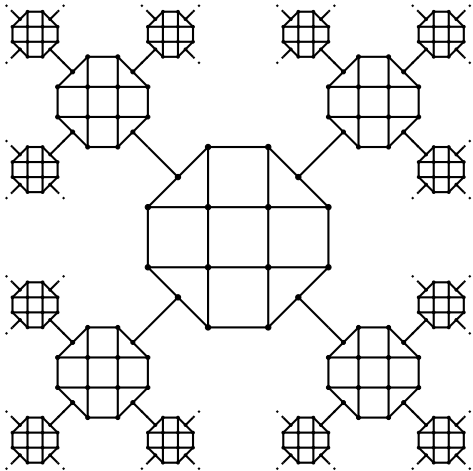

    \centering

    \caption{A non-amenable quasi-transitive graph. The Bernoulli graph of this graph is be symmetrically tileable.}
    \label{fig:nonamenable tileable}
\end{figure}

\section{Constructing the Type Multigraph}\label{sec:typegraph}
We will proceed by constructing a Borel multigraph called the \textbf{type multigraph}, denoted $\Mgraph$.\\

Let $\CF_\circ(n)$ denote the countable set of isomorphism classes of finite rooted graphs with radius $\leq n$. Define $\CF_\circ$ to be a closed subspace of $\prod_{n\in\N}\CF_\circ(n)$ such that $((F_0,r_0),(F_1,r_1)\dots)\in \CF_\circ$ if for all $n$ we have $G_{n+1}\rest B_{n}(r_{n+1}),r_{n+1}) \cong (G_n,r_n)$. To each element $((F_0,r_0),(F_1,r_1)\dots)$ of $\CF_\circ$ we can associate a unique (up to isomorphism) infinite connected locally finite rooted graph $(G,r)$ such that $(F_n,r_n) = (F\rest B_n(r), r)$. In a slight abuse of notation, we will denote elements of $\CF_\circ$ as $(F_0,F_1,\dots, r)$, where $r$ is the root of every $F_n$. We call the associated infinite rooted graph $(F,r)$ the \textbf{limit}  of $(F_0,F_1,\dots,r)$.\\

Similarly, define $\CF_\mathbf{-}(n)$ denote the countable set of isomorphism classes of finite graphs with a distinguished edge $h$ where all vertices are within distance $n$ of one of the endpoints of $h$. Define $\CF_\mathbf{-}$ to be a closed subspace of $\prod_{n\in\omega}\CF_\mathbf{-}(n)$ such that $((F_0,h_0),(F_1,h_1)\dots)\in \CF_\mathbf{-}$ if for every $n$, we have $(F_{n+1}\rest B_{n}(h_{n+1}),h_{n+1}) \cong (F_n,h_n)$. In the same way we did for $\CF_\circ$, we can associate to each element $((F_0,h_0),(F_1,h_1)\dots)$ of $\CF_\mathbf{-}$ a unique (up to isomorphism) \textbf{limit} $(F,h)$.\\
    
For each $e = (F_0,\dots, h)\in \CF_\mathbf{-}$ fix embeddings $f_n:(F_n,h)\to (F_{n+1},h)$. Note that $F_0$ is a single edge $h$ between endpoints $u_0$,$v_0$, then recursively set $u_{n+1}=f_n(u_n)$, $v_{n+1}=f_n(v_n)$. Then to each $e\in \CF_\mathbf{-}$ we can associate two elements of $\CF_\circ$: $u = ((F_n\rest B_n(u_n),u_n))_{n\in\N}$ and $v = ((F_n\rest B_n(v_n),v_n))_{n\in\N}$. This $ u$ and $ v$ do not depend on our specific choice of $f_n$. Define $\ep:\CF_\mathbf{-}\to [\CF_\circ]^2\cup[\CF_\circ]^1$ so that $\ep((F_0,\dots, h)) = \{u,v\}$.\\ 

Equip $\CF_\circ$ and $\CF_\mathbf{-}$ with the product topology. Then $\ep$ is Borel since the initial segments of $u$ and $v$ depend only on the initial segments of $e$.
    
\begin{definition} Define a Borel multigraph $\Mgraph = (\MgraphV,\MgraphE,\ep)$ where $\MgraphV = \CF_\circ$, $\MgraphE =\CF_\mathbf{-}$, and $\ep$ is defined as above.\\

Note that $M$ has countable multiplicity, as for any $v = (F_0,\dots,r) \in \MgraphV$, the size of $N_\MgraphE(v)$ is at most the degree of $r$ in the limit $(F,r)$ of $v$. For each $v = (F_0,\dots, r) \in \MgraphV$, $e = (H_0,\dots, h) \in \MgraphE$, and $i\in \N$, let $n_i(v, e)$ be the number of edges $f$ in $F_{i+1}$ adjacent to the root $r$ with $(F_{i+1}\rest B_{i}(f),f) \cong (H_i, h)$. Now define a Borel edge weighting $m$ of $\Mgraph$ as follows:
$$m(v, e) = \min_{i\in\N} n_i(v,e).$$
Note then that for $v = (F_0,\dots,r) \in \MgraphV$, the degree of $r$ in the limit of $v$ is exactly $\sum_{e\in N_\MgraphE(v)}m(v,e)$.\\

For any locally finite Borel graph $G = (V,E)$ define Borel maps $\phi_V: V\to \MgraphV$, $\phi_E: E\to\MgraphE$ such that for any $x\in V$ and $h\in E$:

    \begin{align*}
        \phi_V(x) &= (G\rest B_{0}(x), G\rest B_{1}(x),\dots, x)\\
        \phi_E(h) &= (G\rest B_{0}(h), G\rest B_{1}(h),\dots, h)
    \end{align*}
\end{definition}

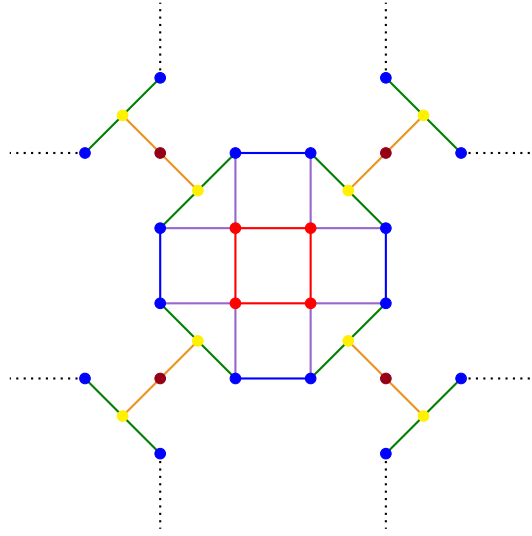
\begin{figure}[ht]
\begin{tikzpicture}

\draw[red, thick] (0,0) -- (0,1);
\draw[red, thick] (0,0) -- (1,0);
\draw[red, thick] (1,0) -- (1,1);
\draw[red, thick] (0,1) -- (1,1);

\draw[amethyst, thick] (1,1) -- (1,2);
\draw[amethyst, thick] (0,1) -- (0,2);
\draw[blue, thick] (1,2) -- (0,2);

\draw[amethyst, thick] (1,1) -- (2,1);
\draw[amethyst, thick] (1,0) -- (2,0);
\draw[blue, thick] (2,1) -- (2,0);

\draw[amethyst, thick] (0,0) -- (0,-1);
\draw[amethyst, thick] (1,0) -- (1,-1);
\draw[blue, thick] (0,-1) -- (1,-1);

\draw[amethyst, thick] (0,0) -- (-1,0);
\draw[amethyst, thick] (0,1) -- (-1,1);
\draw[blue, thick] (-1,0) -- (-1,1);

\draw[ao, thick] (-1,0) -- (0,-1);
\draw[ao, thick] (1,-1) -- (2,0);
\draw[ao, thick] (-1,1) -- (0,2);
\draw[ao, thick] (1,2) -- (2,1);

\draw[carrotorange, thick] (-.5,-.5) -- (-1.5,-1.5);
\draw[carrotorange, thick] (-.5,1.5) -- (-1.5,2.5);
\draw[carrotorange, thick] (1.5,1.5) -- (2.5,2.5);
\draw[carrotorange, thick] (1.5,-.5) -- (2.5,-1.5);

\filldraw[yellow] (-.5,-.5) circle (2pt);
\filldraw[yellow] (-.5,1.5) circle (2pt);
\filldraw[yellow] (1.5,1.5) circle (2pt);
\filldraw[yellow] (1.5,-.5) circle (2pt);

\filldraw[red] (0,0) circle (2pt);
\filldraw[red] (1,0) circle (2pt);
\filldraw[red] (0,1) circle (2pt);
\filldraw[red] (1,1) circle (2pt);

\filldraw[blue] (-1,0) circle (2pt);
\filldraw[blue] (-1,1) circle (2pt);
\filldraw[blue] (2,0) circle (2pt);
\filldraw[blue] (2,1) circle (2pt);
\filldraw[blue] (0,2) circle (2pt);
\filldraw[blue] (1,2) circle (2pt);
\filldraw[blue] (0,-1) circle (2pt);
\filldraw[blue] (1,-1) circle (2pt);

\filldraw[carmine]  (-1,-1) circle (2pt);
\filldraw[carmine] (-1,2) circle (2pt);
\filldraw[carmine] (2,2) circle (2pt);
\filldraw[carmine] (2,-1) circle (2pt);

\draw[ao, thick] (-2,-1) -- (-1,-2);

\draw[black, thick,dotted] (-2,-1) -- (-3,-1);
\draw[black, thick,dotted] (-1,-2) -- (-1,-3);

\draw[ao, thick] (2,-2) -- (3,-1);
\draw[black, thick,dotted] (2,-2) -- (2,-3);
\draw[black, thick,dotted] (3,-1) -- (4,-1);

\draw[ao, thick] (-2,2) -- (-1,3);

\draw[black, thick,dotted] (-2,2)-- (-3,2);
\draw[black, thick,dotted] (-1,3)-- (-1,4);

\draw[ao, thick] (2,3) -- (3,2);
\draw[black, thick,dotted] (2,3)-- (2,4);
\draw[black, thick,dotted] (3,2)-- (4,2);

\filldraw[blue]  (2,3) circle (2pt);
\filldraw[blue]   (3,2) circle (2pt);
\filldraw[blue] (-2,2) circle (2pt);
\filldraw[blue] (-1,3) circle (2pt);

\filldraw[blue] (2,-2) circle (2pt);
\filldraw[blue] (3,-1) circle (2pt);

\filldraw[blue] (-2,-1) circle (2pt);
\filldraw[blue] (-1,-2) circle (2pt);

\filldraw[yellow]  (-1.5,-1.5) circle (2pt);
\filldraw[yellow] (-1.5,2.5) circle (2pt);
\filldraw[yellow] (2.5,2.5) circle (2pt);
\filldraw[yellow] (2.5,-1.5) circle (2pt);

\end{tikzpicture}
\centering
    \caption{The quasi-transitive graph from Figure \ref{fig:amenable_quasi} with edge and vertex automorphism group orbits colored.}
\label{fig:colorquasi}
\end{figure}


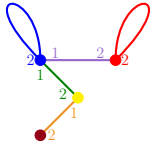
\begin{figure}[ht]
\begin{tikzpicture}

\draw[red, thick] (0,0) .. controls (1,1) and (0,1) .. (0,0)node[anchor=west, scale = .5]{2};
\draw[blue, thick] (-1,0) .. controls (-2,1) and (-1,1) .. (-1,0)node[anchor=east, scale = .5]{2};

\draw[ao, thick]  (-.5,-.5) -- (-1,0);

\node at (-.2,.1)[scale = .5, amethyst]{2};
\node at (-.8,.1)[scale = .5, amethyst]{1};

\node at (-1,-.2)[scale = .5, ao]{1};
\node at (-.7,-.45)[scale = .5, ao]{2};

\node at (-.55,-.7)[scale = .5, carrotorange]{1};
\node at (-.85,-1)[scale = .5, carrotorange]{2};

\draw[amethyst, thick] (0,0) -- (-1,0);
\draw[carrotorange, thick] (-1,-1) -- (-.5,-.5);

\filldraw[red] (0,0) circle (2pt);
\filldraw[blue] (-1,0) circle (2pt);
\filldraw[yellow] (-.5,-.5) circle (2pt);
\filldraw[carmine] (-1,-1) circle (2pt);

\end{tikzpicture}

\caption{The image of the graph in Figure \ref{fig:colorquasi} under the homomorphism $(\phi_V,\phi_E)$.}
\label{fig:colortypegraph}
\end{figure}
In the following claims, let $G$ be a locally finite graph $G = (V,E)$. We will first show that $(\phi_V, \phi_E)$ is a homomorphism from $G$ (viewed as a multigraph) to $\Mgraph$. Moreover, the image of this homomorphism is an induced sub-multigraph of $\Mgraph$ whose connected components are connected components of $\Mgraph$.

\begin{lem}\label{hom_lemma}
    Let $G =  (V,E)$ be a locally finite Borel graph. For any $x,y\in V$, if $x\ E\ y$ then $\ep(\phi_E(x,y)) = \{\phi_V(x),\phi_V(y)\}$, and $\phi_E(N_E(x)) = N_\MgraphE(\phi_V(x))$.
\end{lem}
\begin{proof}

    First, we show that $\phi_V$, $\phi_E$ form a multigraph homomorphism from $G$ viewed as a multigraph to $\Mgraph$. Let $x,y \in V$ with $x\ E\ y$ via an edge $h \in E$. We need to show that $\phi_E(h)$ is an edge from $\phi_V(x)$ to $\phi_V(y)$ in $M$. Recall that $\phi_E(h) = (G\rest B_{0}(h), G\rest B_{1}(h),\dots, h)$. In the definition of  $\ep(\phi_E(h))$ we may take the $f_n:(G\rest B_{n}(h),h)\to (G\rest B_{n+1}(h),h)$ to be the inclusion map. We then have:
    \begin{align*}
        \ep(\phi_E(h)) &= \{(G\rest B_{0}(x), \dots, x), (G\rest B_{0}(y),\dots, y)\}\\
        &= \{\phi_V(x),\phi_V(y)\}
    \end{align*}

    We will now show $\phi_E(N_E(x)) = N_\MgraphE(\phi_V(x))$. By the above we have $\phi_E(N_E(x)) \subseteq N_\MgraphE(\phi_V(x))$, so it suffices to show $\phi_E(N_E(x)) \subseteq N_\MgraphE(\phi_V(x))$. Take any $e = (H_0,H_1,\dots h) \in N_\MgraphE(\phi_V(x))$, then there exists embeddings from $H_n$ to $G$ sending $h$ to an edge incident to $x$. Since $G$ is locally finite, there is some edge $g$ incident to $x$ such that $h$ is sent to $g$ infinitely many times, in which case $(G\rest B_{n}(g),g)\cong (H_n, h)$ for all $n$, so $\phi_E(g) = e$.\\
\end{proof}

The map $\phi_V$ collapses all vertices in the same $\Aut(G)$ orbit. The same goes for edges.

\begin{lem}\label{lem:aut_lemma}
Let $G =(V,E)$ be a locally finite Borel graph.
\begin{itemize}
    \item If $x,y\in V$ with $x\ E\ y$, then $\phi_V(x) = \phi_V(y)$ if and only if there is an automorphism of $G$ taking $u$ to $v$.
    \item If $f,h\in E$ are in the same $G$ component, then $\phi_E(f) = \phi_E(h)$ if and only if there is an automorphism of $G$ taking $f$ to $h$.
\end{itemize}
\end{lem}
\begin{proof}
    If $\phi_V(x) = \phi_V(y)$, then $(G\rest B_{n}(x),x)\cong (G\rest B_{n}(y),y)$ for all $n$. So by compactness, there is an automorphism $f$ of the connected component of $x$ in $G$ sending $x$ to $y$. We can extend this to all of $G$ by letting the automorphism be the identity off the connected component of $x$. The same argument works for edges.
\end{proof}

The map $\phi_E$ may collapse multiple edges incident to $x\in V$ in $G$ to a single edge incident to $\phi_V(x)$ in $\Mgraph$. However, in the next lemma, we will see that the weight of the edge $f \in \MgraphE$ with respect to $\phi_V(x)$ is exactly the number of edges adjacent to $x$ which get collapsed to $f$ by $\phi_E$. Taking this weighting into account, $(\phi_V,\phi_E)$ preserves the weighted degree of each $x\in V$.

\begin{lem}\label{lem:deg_lemma}
Let $G = (V,E)$ be a locally finite Borel graph. Then for any $x\in V(G)$ we have $\deg_G(x) = \wdeg_M(\phi_V(x))$. Moreover for all $x\in V$, $e\in \MgraphE$, we have $m(\phi_V(x),e) = |\{g\in N_E(x): \phi_E(g) = e\}|$.
\end{lem}
\begin{proof}
It is enough to show that $m(\phi_V(x),e) = |\{g\in N_E(x): \phi_E(g) = e\}|$.
If $e = (H_0,H_1,\dots, h)$, then by definition $m(\phi_V(x),e)$ is the minimum over $n$ of the number of edges $g$ in $G \rest B_{n+1}(x)$ adjacent to $x$ with $((G\rest B_{n+1}(x))\rest B_n(g),g)= (G\rest B_n(g),g)\cong (H_n,h)$.\\

Enumerate $\{ g\in N_E(x): \phi_E(g) = e\}$ as $g_1,\dots,g_k$. We then have $(G\rest B_n(g_i),g_i)\cong (H_n,h)$ for all $n$, so $m(e)\geq k$. If for each $n$ there exists $g\in N_E(x)$ with $g \notin \{g_1,\dots,g_k\}$ such that $(G\rest B_n(g),g)\cong (H_n,h)$ for all $n$, then since $G$ is locally finite there exists a $g\in N_E(x)\setminus \{g_1,\dots,g_i\}$ for which $(G\rest B_n(g),g)\cong (H_n,h)$ for infinitely many $n$, in which case $\phi_E(g) = e$, contradicting our choice of $g_i$.
.
\end{proof}

\section{Automorphism Invariant Fractional Perfect Matchings in Amenable quasi-transitive Graphs}\label{qtgraphs}
In this section we will study automorphism invariant fractional perfect matchings in the classical (non-Borel) context. An \textbf{automorphism invariant fractional perfect matching} on $G$ is a fractional perfect matching $f$ on $G$ such that for any $\phi\in \Aut(G)$ and any $e\in E(G)$, we have $f(e) = f(\phi(e))$. In the following proposition, we will produce an automorphism invariant fractional perfect matching on a locally finite amenable graph from a fractional perfect matching. This can be thought of as integrating the fractional matching with respect to the mean on the automorphism group. We present a more detailed proof using our definition of amenability.

\begin{prop}\label{invar}
    Let $G$ be a locally finite amenable quasi-transitive graph. If there exists a fractional perfect matching on $G$, then there exists an automorphism invariant fractional perfect matching on $G$.
\end{prop}
\begin{propositionproof}
    Let $f$ be a fractional perfect matching on $G$. Let $\mu$ be the Haar measure on $\Aut(G)$. Take $C_n$ to be an increasing sequence of compact subsets of $\Aut(G)$ such that $\bigcup C_n = \Aut(G)$ and the identity function $\Id:G\to G$ is contained in $C_0$. 

    \begin{numclaim} Take $\eps_n > 0$ be any positive sequence decreasing to $0$. Then there exists an increasing sequence of compact sets $B_n$ with $C_n\subseteq B_n$ and $\frac{\mu(B_n B_{n+1}\Delta B_{n+1})}{\mu(B_{n+1})} <\eps_n$.
    \end{numclaim}
    \begin{proof}
    Let $B_0 = C_0$. Assume $B_n$ is defined, then by amenability of $\Aut(G)$, let $B'_{n+1}\subseteq \Aut(G)$ be a compact set with: 
    $$\frac{\mu( B_n B_nC_{n+1}B'_{n+1}\Delta B'_{n+1})}{\mu(B'_{n+1})} < \eps_n$$
    Then take $B_{n+1} = B_nC_{n+1}B'_{n+1}$. Note that:
    $$\frac{\mu( B_nB_{n+1})}{\mu(B_{n+1})} \leq \frac{\mu( B_nB_{n+1})}{\mu(B'_{n+1})} < \eps_n$$
    And the inclusions $B_n, C_{n+1}\subseteq B_{n+1}$ hold since $\Id \in C_0$.
\end{proof}

Fix a positive decreasing sequence $\eps_n\to 0$ and $B_n$ as in the above claim.\\

We can define an action $\Aut(G)$ on the set on fractional matchings of $G$ via the diagonal action $\gamma \cdot f ((u,v)) = f((\gamma(u),\gamma(v)))$. We can now define the automorphism invariant fractional perfect matching $f'$.
$$
f'(e) =\LIM_{n} \frac{1}{\mu(B_n)}\int_{B_n} \gamma\cdot f(e) d\mu(\gamma)
$$
This is a fractional perfect matching, as:
\begin{align*}
   \sum_{e\in N_E(v)}f'(e) &= \LIM_{n} \frac{1}{\mu(B_n)}\int_{B_n} \sum_{e\in N_E(v)}\gamma\cdot f(e) d\mu(\gamma)\\
   &= \LIM_{n} \frac{1}{\mu(B_n)}\int_{B_n} \sum_{e\in N_E(\gamma v)} f(e)\\
   &= \LIM_{n} \frac{1}{\mu(B_n)}\int_{B_n}1 d\mu(\gamma)\\
   &= 1\\
\end{align*}
It is automorphism invariant, since for every $\eta\in\Aut(G)$, there exists an $m\in\N$ such that
$\eta \in B_m$ and we have:
\begin{align*}
   |f'(e) - \eta\cdot f'(e)| &= \LIM_{n} \frac{1}{\mu(B_n)}|(\int_{B_n}\gamma\cdot f(e)d\mu(\gamma) - \int_{B_n}\gamma\eta\cdot f(e)d\mu(\gamma))|\\
   &= \LIM_{n} \frac{1}{\mu(B_n)}(|\int_{B_n}\gamma\cdot f(e)d\mu(\gamma) - \int_{\eta B_n}\gamma\cdot f(e)d\mu(\gamma)|)\\
   &\leq \LIM_{n} \frac{1}{\mu(B_n)}\mu(B_n\Delta \eta B_n)\\
   &\leq \LIM_{n} \frac{1}{\mu(B_n)}\mu(B_n\Delta B_{m} B_n)\\
   &\leq \LIM_{n\geq m} \frac{\mu(B_{n+1}\Delta B_{n} B_{n+1})}{\mu(B_{n+1})}\\
   &\leq \LIM_{n\geq m} \eps_n = 0
\end{align*}
\end{propositionproof}

Applying this proposition componentwise to an amenable quasi-transitive Borel graph $G$ shows that in any such graph, there exists a fractional perfect matching that is invariant under automorphisms of its components.\\

\section{Borel Fractional Perfect Matchings in Amenable quasi-transitive Graphs}\label{qtgraphs}

In the next proposition, we will show that for a locally finite Borel graph $G$  where the image of $G$ under the homomorphism $(\phi_V,\phi_E)$ has a Borel weighted fractional perfect matching (as an induced subgraph of $\Mgraph$) then that weighted fractional perfect matching can be pulled back over $\phi_E$ to a Borel fractional perfect matching of $G$.

\begin{prop}\label{pullback}
    Let $G$ be a locally finite Borel graph. If there exists a Borel fractional perfect matching on the weighted multigraph $\Mgraph' = (\Image(\phi_V),\Image(\phi_E),\ep\rest \Image(\phi_E), m\rest \Image(\phi_E))$, then there is a Borel fractional perfect matching of $G$.
\end{prop}
\begin{proof}
    Let $f'$ be the Borel fractional perfect matching on $\Mgraph'$. Take $f = f'\circ \phi_E$, then $f'$ is a Borel fractional perfect matching on $G$, since for any $v\in V$ we have:
    \begin{align*}
        \sum_{h\in N_E(v)}g(e) &= \sum_{e \in \phi_E(N_E(v))}\sum_{h\in N_E(v)\cap\phi_E^{-1}(e)}g(h)\\
        &= \sum_{e \in \phi_E(N_E(v))}\sum_{h\in N_E(v)\cap\phi_E^{-1}(e)}g'(\phi_E(h))
    \end{align*}
    By Lemma \ref{hom_lemma} we have that $N_E(\phi_V(v))=\phi_E(N_E(v))$.
    \begin{align*}
        &\sum_{e \in \phi_E(N_E(v))}\sum_{h\in N_E(v)\cap\phi_E^{-1}(e)}g'(\phi_E(h))\\&= \sum_{e \in N_E(\phi_V(v))}\sum_{h\in N_E(v)\cap\phi_E^{-1}(e)}g'(\phi_E(e))\\
        &= \sum_{e \in N_E(\phi_V(v))}g'(\phi_E(e))|N_E(v)\cap\phi_E^{-1}(e)|
    \end{align*}
    By Lemma \ref{lem:deg_lemma} we have that $|N_E(v)\cap\phi_E^{-1}(e)| = m(\phi_V(v),e)$.
    \begin{align*}
    &\sum_{e \in N_E(\phi_V(v))}g'(\phi_E(e))|N_E(v)\cap\phi_E^{-1}(e)|\\
        &= \sum_{e \in  N_E(\phi_V(v))}g'(e)m(\phi_V(v),e)\\
        &= 1
    \end{align*}
\end{proof}
In the next proposition, we will show that if $G$ is componentwise quasi-transitive and each component admits an automorphism invariant fractional perfect matching, then $G$ admits a Borel fractional perfect matching. We do this by showing that for such a $G$, the image of $(\phi_G,\phi_E)$ is always mapped to a part of $\Mgraph$ that has a Borel weighted fractional perfect matching. 

\begin{prop}\label{main_prop}
    Let $G$ be a locally finite Borel graph. If $G$ is componentwise quasi-transitive, and each component admits an automorphism invariant fractional perfect matching, then $G$ admits a Borel fractional perfect matching.
\end{prop}

\begin{propositionproof} 
    Let $M$ be the type multigraph. Take $F\subseteq \MgraphE$ to be the Borel subset of all $e\in\MgraphE$ whose component in $M$ is finite. Let $F'\subseteq F$ be the set of edges in $M$ whose components in $M$ are finite and admit a weighted fractional perfect matching with respect to the weighting induced on $M$. Note that $F'$ is a union of components of $M$. We can treat $F'$ as a sub-multigraph of $M$, with weighting induced by $m$.\\ 
    
    \begin{numclaim}\label{F'claim}
        $F'$ is Borel and there exists a Borel weighted fractional perfect matching on $F'$.
    \end{numclaim}
    \begin{proof}

    The equivalence relation whose classes are connected components of $M$ is smooth on $F$, so let $T\subseteq F$ be a transversal and $s:F\to T$ be a selector. Now consider $R$ consisting of $(e,g)$ where $e\in F$ and $g$ is a weighted fractional perfect matching on the edge component of $e$.\\

    Since the set of weighted fractional perfect matchings for a finite weighted multigraph graph is a polytope, the section above each $e$ in $R$ is compact, so $F' = \proj_F(R)$ is Borel and there is a Borel uniformization $u$ of $R$. For any $e\in F'$, $u(e)$ is a weighted fractional perfect matching on the component of $e$. We can then define a Borel weighted fractional perfect matching $f': F' \to [0,1]$ via $f'(e) = u(s(e))(e)$.
    
    \end{proof}

    Let $G = (V,E)$ be a locally finite Borel graph where each component is quasi-transitive and admits an automorphism invariant fractional perfect matching. Recall that $(\phi_V,\phi_E)$ denote the previously constructed multigraph homomorphism from $G$ to $M$.
    
    \begin{numclaim}\label{phiclaim}
        $\phi_E(E(G)) \subseteq F'$
    \end{numclaim}
    \begin{proof}
         Since $G$ is componentwise quasi-transitive, $\phi_E(h) \in F$ for every $h\in E(G)$. 
         Fix $h\in E(G)$, let $C$ be component of $h$ in $G$. Let $C'$ be the component of $\phi_E(h)$ in $M$. We will show that $\phi_E(h) \in F'$ by constructing a weighted fractional perfect matching on $C'$. \\
         
         Let $g$ be an automorphism invariant fractional perfect matching $C$. Define a function $g'$ on $E(C')$ as follows: 
         for every edge $e$ in $C'$, by Lemma \ref{hom_lemma} there exists an edge $d$ in $E(C)$ such that $\phi_E(d) = e$, then define $g'(e)$ to be $g(d)$. We need to show $g'$ is well defined. Suppose $d,c$ in $E(C)$ with $\phi_E(d) = \phi_E(c)$, then by Lemma \ref{lem:aut_lemma} we have that there is an automorphism of the component of $h$ taking $d$ to $c$. Since $g$ is automorphism invariant, $g(d) = g(c)$.\\

         Now we will show $g'$ is a weighted fractional perfect matching on $C'$. Take any vertex $u$ in $V(C')$, then by Lemma \ref{hom_lemma} there exists $x\in V(C)$ with $\phi_V(x) = u$. Then:
      \begin{align*}
        \sum_{e\in N_E(u)}g'(e)m(u,e)&= \sum_{e\in N_E(\phi_V(x))}g'(e)m(\phi_V(x),e)
    \end{align*}
    By Lemma \ref{lem:deg_lemma}, we have that $m(\phi_V(x),e) = |N_E(x)\cap\phi_E^{-1}(e)|$, so we have:
    \begin{align*}
        \sum_{e\in N_E(\phi_V(x))}g'(e)m(\phi_V(x),e)&=\sum_{e\in N_E(\phi_V(x))}g'(e)|N_E(x)\cap\phi_E^{-1}(e)|\\
        &=\sum_{e\in N_E(\phi_V(x))}\sum_{d\in N_E(x)\cap\phi_E^{-1}(e)}g'(e)\\
        &=\sum_{d\in N_E(x)}g'(\phi_E(d))\\
        &= \sum_{d\in N_E(x)}g(d)\\
        &= 1
    \end{align*}
    This shows $g'$ induces a weighted fractional perfect matching on $C'$, so $\phi_E(h) \in F'$.\\
    \end{proof}

    By Claim \ref{phiclaim} we have have $\phi_E(E(G)) \subseteq F'$. By Claim \ref{F'claim} there exists Borel weighted fractional perfect matching $f'$ on $F'$, and by Lemma \ref{hom_lemma} we have that $f'\rest \phi_E(E(G))$ is a Borel weighted fractional perfect matching of the induced weighted multigraph on $\phi_E(E(G))$. We can now apply Proposition \ref{pullback}, to get a Borel fractional perfect matching on $G$.
\end{propositionproof}
Combining Propositions \ref{main_prop} and \ref{invar}, we get Theorem \ref{thm:mainthm}.

\section{Fractional Matchings in Borel Symmetrically Tileable Graphs}\label{stgraphs}

Now we will look at symmetrically tileable graphs. Unlike in the quasi-transitive context, we can produce a Borel fractional matching that is bounded away from $0$ on the matching kernel. In the proof below we will average over an automorphism group like in Proposition \ref{invar}, but instead of the automorphism group of the entire component we average over the automorphism group of the tile. Below we prove a slightly stronger version of Theorem $\ref{thm:tileable_graph}$.

\begin{prop} \label{prop:tileable_graph}
    If $G$ is a locally finite Borel bipartite graph satisfying Halls theorem, and $G$ is Borel symmetrical tilable, then $G$ admits a Borel fractional perfect matching $f$. Moreover, there exists $\theta > 0$ such that any edge $e$ in the perfect matching kernel satisfies $\theta<f(e)<1-\theta$.
\end{prop}

\begin{proof}
Let $G = (V,E)$ locally finite Borel bipartite graph on standard Borel space $V$ which satisfies Hall's condition and is Borel symmetrically tileable. Let $F = (V_F,E_F,H_F,a)$ be a half-edge transitive finite graph with half-edges witnessing that $G$ is Borel symmetrically tilable.\\

First, we construct a fractional perfect matching $\tau'$ of $F$ such that any edge/half-edge $e$ we have $\tau'(e) = 0$ if and only if no matching of $F$ contains $e$. For each edge, $e\in E_F\cup H_F$ let $M_{e}$ be a perfect matching of $F$ which includes $e$ if such a matching exists, otherwise let $M_{e}$ be an arbitrary perfect matching of $F$. Such an $M_e$ will always exist since $F$ is the restriction of a graph satisfying Hall's theorem. Let
$$\tau' = \frac{1}{|E_F\cup H_F|} \sum_{e\in E_F\cup H_F} M_e$$
Since $\tau'$ is a convex combination of perfect matchings, it is a fractional perfect matching.\\

Next, we average $\tau'$ over all automorphisms of $F$, namley let

$$\tau = \frac{1}{|\Aut(F)|} \sum_{\gamma \in \Aut(F)} \tau'\circ\gamma$$Each $\tau'\circ\gamma$ is a fractional perfect matching, so $\tau$ is a fractional perfect matching also. Moreover, $\tau$ is $\Aut(F)$-invariant, since for every $\beta \in \Aut(F)$, $\tau\circ \beta = \frac{1}{|\Aut(F)|} \sum_{\gamma \in \Aut(F)} \tau'\circ\gamma\circ\beta = \frac{1}{|\Aut(F)|} \sum_{\gamma \in \Aut(F)} \tau'\circ\gamma = \tau$. Since $F$ is half-edge transitive, every two half-edges get mapped to one another via an automorphism, so $\tau$ is constant on the half-edges. Let $c$ be that constant.\\

 Now take $\CT\subseteq [V]^{<\omega}$ to be a Borel $F$-tiling of $G$. For each tile $T\in\CT$, fix an isomorphism $i_T:T\to F$. Define fractional perfect matching $\eta:V\to [0,1]$ on $V$ such that for any $e = (x,y)\in E$, if $x$ and $y$ are in the same tile $T\in \CT$, then define $$\eta(e) = \tau(i_T(x),i_T(y))$$ Otherwise,  if $x$ and $y$ are adjacent but not in the same tile, let $$\eta(e) = c$$Then, for any $x\in V$ with $x\in T\in\CT$, $\eta \rest N_{\frac12}(x) = \tau \rest N_{\frac12}(i_T(x))$, so $\eta$ satisfies $\sum_{e\in N_E(x)} \eta(e)= \sum_{e\in N_{E_0}(x)} \tau(e) = 1$, so it is a fractional perfect matching.\\

Finally, we claim there exists $\theta \in (0,1)$ such that for every $e\in E(G)$ in the matching kernel, $\theta <\eta(e)< 1-\theta$ . Since $\tau$ only takes on finitely many values, it suffices to show that $\eta(e) = 0$ then there are no matchings of $G$ which contain $e$. Let $T$ be any tile containing $e$ either as an edge or half-edge. If there is a matching $M$ of $G$ containing $e$, then there is a matching of $G\rest N_{\frac12}(T)$ which is isomorphic to $F$ via $i_T$. This means that the edge $g$ in $F$ corresponding to $e$ in $T$ must have $\tau'(g) >0$, so $\tau(g)>0$, so $\eta(e) = \tau(g) >0$.
\end{proof}
The following corollary can be deduced from the results in \cite{bowen2022perfect}, and it also follows directly from from Proposition \ref{thm:tileable_graph} and the main result in \cite{bowen2022perfect}.
\begin{cor}
    If $G$ is a locally finite Borel bipartite graph satisfying Halls theorem, and $G$ is Borel symmetrical tilable, then $G$ admits a Measurable perfect matching.
\end{cor}

\bibliographystyle{alpha}
\bibliography{bibliography}

\end{document}